\documentclass{amsart}
\usepackage{amsmath}
\usepackage{amsfonts}
\usepackage{amssymb}
\usepackage{mathtools}
\usepackage{thmtools}
\usepackage{thm-restate}
\usepackage{epsfig}
\usepackage{tikz}
\usepackage{multirow}
\usepackage{url}
\usepackage{array}
\usepackage{float}
\usepackage{tikz-cd}
\usepackage{hyperref}

\newtheorem{theorem}{Theorem}[section]
\newtheorem{prop}[theorem]{Proposition}
\newtheorem{lemma}[theorem]{Lemma}
\newtheorem{cor}[theorem]{Corollary}

\theoremstyle{definition}
\newtheorem{rem}[theorem]{Remark}
\newtheorem{defi}[theorem]{Definition}
\newtheorem{example}[theorem]{Example}

\DeclarePairedDelimiter{\ceil}{\lceil}{\rceil}

\newcommand{\ra}{\rightarrow}
\newcommand{\IP}{\mathbb{P}}
\newcommand{\IC }{\mathbb{C}}
\newcommand{\IR}{\mathbb{R}}

\newcommand{\IZ}{\mathbb{Z}}
\newcommand{\IQ}{\mathbb{Q}}
\newcommand{\IN}{\mathbb{N}}

\newcommand{\coloneqq}{:=}

\newcommand{\spinn}{\mathrm{sn}_{\mathbb{R}}}

\newcommand{\hskt}{S^{\left[3\right]}}

\newcommand{\hskn}{S^{\left[n\right]}}

\newcommand{\sigman}{\Sigma^{\left[n\right]}}
\newcommand{\pnine}{P_{8t}(9)}
\newcommand{\ptwelve}{P_{8t}(12)}

\newcommand{\pone}{P_{8t}(1)}

\DeclareMathOperator{\Pic}{Pic}

\DeclareMathOperator{\id}{id}

\DeclareMathOperator{\divi}{div}

\DeclareMathOperator{\Mo}{Mon}
\DeclareMathOperator{\aut}{Aut}
\DeclareMathOperator{\bir}{Bir}

\DeclareMathOperator{\ns}{NS}

\DeclareMathOperator{\pic}{Pic}

\DeclareMathOperator{\Sym}{Sym}
\DeclareMathOperator{\trans}{Tr}
\DeclareMathOperator{\movb}{Mov}
\DeclareMathOperator{\nef}{Nef}

\makeatletter
\def\blfootnote{\xdef\@thefnmark{}\@footnotetext}

\begin{document}

\title[On birational transformations of Hilbert schemes of points on K3's]{On birational transformations of Hilbert schemes of points on K3 surfaces}

\author[P.\ Beri]{Pietro Beri}
\address{Pietro Beri, Laboratoire de Math\'ematiques et Applications, UMR CNRS 7348, Universit\'e de Poitiers, T\'el\'eport 2, Boulevard Marie et Pierre Curie, 86962 Futuroscope Chasseneuil Cedex, France.}
\email{pietro.beri@math.univ-poitiers.fr}

\author[A.\ Cattaneo]{Alberto Cattaneo}

\address{Alberto Cattaneo, Mathematisches Institut and Hausdorff Center for Mathematics, Universit\"at Bonn, Endenicher Allee 60, 53115 Bonn, Germany; Max Planck Institute for Mathematics, Vivatsgasse 7, 53111 Bonn, Germany.}
\email{cattaneo@math.uni-bonn.de}

\maketitle


\begin{abstract}
We classify the group of birational automorphisms of Hilbert schemes of points on algebraic K3 surfaces of Picard rank one. We study whether these automorphisms are symplectic or non-symplectic and if there exists a hyperk\"ahler birational model on which they become biregular. We also present new geometrical constructions for some of these automorphisms.
\end{abstract}

\section{Introduction} \label{sec: intro}

Hilbert schemes of points on smooth complex K3 surfaces are classical examples of irreducible holomorphic symplectic (ihs) manifolds, i.e.\ smooth compact simply connected K\"ahler manifolds with a unique (up to scalar) holomorphic two-form, which is everywhere non-degenerate.
The second integral cohomology group of ihs manifolds admits a lattice structure, provided by the Beauville-Bogomolov-Fujiki (BBF) quadratic form. The global Torelli theorem for K3 surfaces has been generalized to ihs manifolds (\cite{verb, huybr_torelli, markman}), albeit in a weaker form, making possible to investigate automorphisms of these manifolds by studying their action on the BBF lattice. 

For $t \in \IN$ consider a complex algebraic K3 surface $S$ whose Picard group is generated by an ample line bundle $H$ with $H^2 = 2t$, i.e.\ a very general element of the $19$-dimensional space of $2t$-polarized K3 surfaces. A classification of the group of biregular automorphisms of $\hskn \coloneqq \text{Hilb}^n(S)$ has been given by Boissi{\`e}re, An.\ Cattaneo, Nieper-Wi{\ss}kirchen and Sarti \cite{bcnws} in the case $n=2$, and by the second author for all $n \geq 3$  \cite{catt_autom_hilb}. In particular, $\aut(\hskn)$ is either trivial or generated by an involution which is non-symplectic, i.e.\ it acts as $-1$ on the generator of $H^{2,0}(\hskn)$. 
The results of \cite{bcnws} and \cite{catt_autom_hilb} provide explicit numerical conditions on the parameters $n$ and $t$ to determine whether $\aut(\hskn) = \left\{ \id \right\}$ or $\aut(\hskn) \cong \IZ / 2\IZ$. The first aim of the paper is to give a similar classification for the group $\bir(\hskn)$ of birational automorphisms of the Hilbert scheme, thus generalizing the results of Debarre and Macr\`i \cite{debarre_macri, debarre} for $n=2$. 
We do this in Section \ref{sec: classification}, by combining 
Markman's results on monodromy operators and the chamber decomposition of the movable cone for
ihs manifolds \cite{markman} with the computation of the extremal rays of the movable cone of $S^{[n]}$ due to Bayer and Macr\`i \cite{bayer_macri_mmp}.
These computation translate naturally in the language of Pell's equations, whose basic theory is recalled in Appendix \ref{sec: prelim}. For $\sigma \in \bir(\hskn)$ let $\sigma^* \in O(H^2(\hskn,\IZ))$ be its pullback action on the BBF lattice. We set $H^2(\hskn,\IZ)^{\sigma^*} \coloneqq \left\{ v \in H^2(\hskn,\IZ) \mid \sigma^*(v) = v \right\}$ and $H^2(\hskn,\IZ)_{\sigma^*} \coloneqq (H^2(\hskn,\IZ)^{\sigma^*})^\perp$. The classification of $\bir(\hskn)$ is as follows.

\begin{theorem}\label{thm: iff}
Let $S$ be an algebraic K3 surface with $\pic(S) = \IZ H$, $H^2 = 2t$ and $n \geq 2$.

If $t \geq 2$, there exists a non-trivial birational automorphism $\sigma: \hskn \dashrightarrow \hskn$ if and only if 
$t(n-1)$ is not a square and the minimal solution $(X,Y) = (z,w)$ of Pell's equation $X^2 - t(n-1)Y^2 = 1$ with $z \equiv \pm 1 \pmod{n-1}$ satisfies $w \equiv 0 \pmod{2}$ and $(z, z) \equiv (j, k) \in \frac{\IZ}{2(n-1)\IZ} \times \frac{\IZ}{2t \IZ}$ with $(j, k) \in \left\{ (1,1), (1,-1), (-1,-1)\right\}$. If so, $\sigma$ generates $\bir(\hskn) \cong \IZ/2\IZ$ and
\begin{itemize}
\item if $(j,k) = (1,-1)$, then $\sigma$ is non-symplectic with $H^2(\hskn,\IZ)^{\sigma^*} \cong \langle 2\rangle$;
\item if $(j,k) = (-1,-1)$, then $\sigma$ is non-symplectic with $H^2(\hskn,\IZ)^{\sigma^*} \cong \langle 2(n-1)\rangle$;
\item if $(j,k) = (1,1)$, then $\sigma$ is symplectic with $H^2(\hskn,\IZ)_{\sigma^*} \cong \langle -2(n-1)\rangle$.
\end{itemize}

If $t = 1$, let $(X,Y) = (a,b)$ be the integer solution of $(n-1)X^2 - Y^2 = -1$ with smallest $a, b > 0$. If $n-1$ is a square or $b \equiv \pm 1 \pmod{n-1}$, then $\bir(\hskn) = \aut(\hskn) \cong \IZ/2\IZ$. Otherwise $n \geq 9$, $\bir(\hskn) \cong \IZ/2\IZ \times \IZ/2\IZ$ and $\aut(\hskn) \cong \IZ/2\IZ$.
\end{theorem}

It is interesting to notice that, differently from the biregular case, birational involutions of $\hskn$ can be symplectic, for (infinitely many) suitable choices of $n, t$. In Proposition \ref{prop:families} we provide examples of sequences of degrees $t = t_k(n)$ which realize all cases in Theorem \ref{thm: iff}.
By a classical result of Saint-Donat \cite{saint-donat}, the case $t=1$ is the only one where the K3 surface $S$ has a non-trivial automorphism ($S$ is a double cover of $\mathbb{P}^2$ ramified over a sextic curve), hence $\aut(\hskn) \cong \IZ/2\IZ$ is generated by the corresponding natural (non-symplectic) involution for all $n \geq 2$. 

By \cite[\S 5.2]{markman}, the closed movable cone of a projective ihs manifold $X$ is equipped with a decomposition in K\"ahler-type chambers
\begin{equation} \label{eq: wall-and-ch}
\overline{\movb(X)} = \bigcup_g g^* \nef(X')
\end{equation}
\noindent where the union is taken over all non-isomorphic ihs birational models $g: X \dashrightarrow X'$. The chambers $g^* \nef(X')$ have pairwise disjoint interiors and are permuted by the action of any birational automorphism of $X$; biregular automorphisms are exactly those which map $\nef(X)$ to itself. If $X$ admits a birational automorphism $\sigma$, it is then natural to ask whether there exists an ihs birational model $g: X \dashrightarrow X'$ such that $g \circ \sigma \circ g^{-1}$ is biregular on $X'$. We confront this problem for $X = \hskn$ in Section \ref{subs: decomposition}, by using the explicit description of the walls between chambers coming from \cite{bayer_macri_mmp}. For a fixed $n \geq 2$, we say that a value $t \geq 1$ is \textit{$n$-irregular} if, for a $2t$-polarized K3 surface $S$ of Picard rank one, the group $\bir(S^{[n]})$ contains an involution which is not biregular on any ihs birational model of $\hskn$. We provide a numerical characterization of $n$-irregular values in Propositions \ref{prop: t=1 e id} and \ref{prop: ambiguous}. In particular, we verify that symplectic birational involutions remain strictly birational on all birational models. On the other hand, in the case of non-symplectic birational automorphisms there are only finitely many $n$-irregular $t$'s for a fixed $n \geq 2$ (Corollary \ref{cor: finite}).

Section \ref{sec: n=3} provides an in-depth analysis of the case $n=3$. We show that, when $\bir(\hskt)$ is not trivial, the number of chambers in the movable cone of $\hskt$ (i.e.\ the number of non-isomorphic ihs birational models) is either one, two, three or five. For $n=2$ it is known that the number of chambers in $\movb(S^{[2]})$ is at most three. As $n$ increases, computing an upper bound for the number of chambers becomes more and more difficult, since walls arise from the solutions of an increasing number of Pell's equations. 

In Section \ref{sec: ambig} we explain how Theorem \ref{thm: iff} can be generalized to study birational maps between Hilbert schemes of points on two distinct K3 surfaces of Picard rank one (which need to be Fourier-Mukai partners). In the literature, isomorphic Hilbert schemes of points on two non-isomorphic K3 surfaces are called (strongly) ambiguous. The main result of the section is Theorem \ref{thm: ambig}, an improved version of the criterion \cite[Theorem 2.2]{mmy} for the determination of ambiguous pairs (up to isomorphism or birational equivalence).


Finally, in Section \ref{sec: constructions} we discuss geometrical constructions of the automorphisms arising from Theorem \ref{thm: iff}, adding some new examples to the existing literature. Finding similar constructions is a non-trivial problem, since the Torelli-like results used in the proof of Theorem \ref{thm: iff} give no insight on the geometry. 

\textbf{Acknowledgements.} The authors are indebted to Alessandra Sarti for reading a first draft of the paper and for her precious remarks. This work has greatly benefited from discussions with Samuel Boissi{\`e}re, Chiara Camere and Georg Oberdieck. A.~C. is grateful to Max Planck Institute for Mathematics in Bonn for its hospitality and financial support. A.~C. is supported by the Deutsche Forschungsgemeinschaft (DFG, German Research Foundation) under Germany's Excellence Strategy -- GZ 2047/1, Projekt-ID 390685813.

\section{Birational automorphisms of $\hskn$} \label{sec: classification}

Let $S$ be an algebraic K3 surface with $\pic(S) = \IZ H$, $H^2 = 2t$, $t \geq 1$. For $n \geq 2$, let $\hskn$ be the Hilbert scheme of $n$ points on $S$ and $\left\{ h, -\delta \right\}$ a basis for $\ns(\hskn) \subset H^2(\hskn, \IZ)$, where $h$ is the class of the nef (not ample) line bundle induced by $H$ on $\hskn$ and $2\delta$ is the class of the exceptional divisor of the Hilbert--Chow morphism $\hskn \ra S^{(n)}$. We consider $H^2(\hskn, \IZ)$ equipped with the even lattice structure given by the Beauville--Bogomolov--Fujiki integral quadratic form \cite[Proposition 6]{beauville}. In particular, $H^2(\hskn, \IZ) \cong H^2(S, \IZ) \oplus \IZ \delta \cong U^{\oplus 3} \oplus E_8(-1)^{\oplus 2} \oplus \langle -2(n-1)\rangle$, where for an integer $d \neq 0$ we denote by $\langle d \rangle$ the rank one lattice generated by an element of square $d$.

\subsection{The action on cohomology} \label{subs: action on cohom}

By \cite[Corollary 5.2]{oguiso} the group $\bir(\hskn)$ is finite and the homomorphism $\bir(\hskn) \ra O(H^2(\hskn,\IZ)), \sigma \mapsto \sigma^*$ is injective by \cite[Proposition 10]{beauville_rmks}. Moreover, the kernel of $\Psi: \bir(\hskn) \ra O(\ns(\hskn))$, $\sigma \mapsto \sigma^*\vert_{\ns(\hskn)}$, is the subgroup of natural automorphisms, which is isomorphic to $\aut(S)$ (see \cite[Lemma 2.4]{bcnws} and notice that $\ker(\Psi) \subset \aut(\hskn)$ by the global Torelli theorem \cite[Theorem 1.3]{markman}). By \cite[\S 5]{saint-donat}, this implies that $\ker(\Psi) = \left\{ \id \right\}$ if $t \geq 2$, while for $t = 1$ we have $\ker(\Psi) = \langle \iota^{[n]} \rangle$, where $\iota$ is the covering involution which generates $\aut(S)$ and $\iota^{[n]}$ is the natural involution induced by $\iota$ on $\hskn$. Let $\movb(\hskn) \subset \ns(\hskn)_\IR$ be the movable cone of $\hskn$, i.e.\ the open cone generated by classes of divisors whose base locus has codimension at least two. It contains the ample cone $\mathcal{A}_{\hskn}$. If there exists $\sigma \in \bir(\hskn)$ non-natural, then the (non-trivial) action $\sigma^*$ on $\ns(\hskn)$ preserves the movable cone by \cite[Lemma 6.22]{markman}.



We refer to to Appendix \ref{sec: prelim} for notation and results on Pell's equations, which we will use throughout the paper.

\begin{prop} \cite[Proposition 13.1]{bayer_macri_mmp}, \cite[Proposition 6.1]{catt_autom_hilb} \label{prop: action}
Let $S$ be an algebraic K3 surface with $\pic(S) = \IZ H$, $H^2 = 2t$, $t \geq 1$. For $n \geq 2$, let $\sigma \in \bir(\hskn)$ be a non-natural automorphism. Then $t(n-1)$ is not a square, $(n-1)X^2 - tY^2 = 1$ has no integer solutions (if $n \neq 2$) and $\overline{\movb(\hskn)} = \langle h, zh-tw\delta\rangle_{\IR_{\geq 0}}$, where $(z,w)$ is the minimal solution of Pell's equation $P_{t(n-1)}(1)$ with $z \equiv \pm 1 \pmod{n-1}$. Moreover $z \equiv \pm 1 \pmod{2(n-1)}$ and $w \equiv 0 \pmod{2}$. The isometry $\sigma^* \in O(\ns(\hskn))$ is the reflection of $\ns(\hskn)$ which fixes the line spanned by $(n-1)wh - (z-1)\delta$. With respect to the basis $\left\{ h, -\delta \right\}$ of $\ns(\hskn)$, its matrix is 
\begin{equation} \label{eq: action}
 \begin{pmatrix}
z & -(n-1)w \\ tw & -z
\end{pmatrix}.
\end{equation}
\end{prop}

Thus, assuming that there exists $\sigma \in \bir(\hskn)$ non-natural, we have $\bir(\hskn) = \langle \sigma \rangle \cong \IZ/2\IZ$ if $t \geq 2$, while $\bir(\hskn) = \langle \iota^{[n]}, \sigma \rangle \cong \IZ/2\IZ \times \IZ/2\IZ$ if $t=1$.  

In the following, for any even lattice $L$ with bilinear form $(-,-): L \times L \ra \IZ$ we denote by $A_L \coloneqq L^\vee / L$ the discriminant group equipped with the finite quadratic form $q_L: A_L \ra \IQ / 2\IZ$ induced by the quadratic form on $L$. If $A_L$ is cyclic of order $m$, we write $A_L \cong \frac{\IZ}{m \IZ} \left( \alpha \right)$ if $q_L$ takes value $\alpha \in \IQ / 2\IZ$ on a generator of $A_L$. For $g \in O(L)$ we denote by $\overline{g}$ the isometry induced by $g$ on $A_L$. In the case of the lattice $H^2(\hskn, \IZ)$ we have $A_{H^2(\hskn, \IZ)} \cong \frac{\IZ}{2(n-1) \IZ} \left( -\frac{1}{2(n-1)} \right)$. For any $x \in H^2(\hskn,\IZ)$ let $\divi(x)$ be the divisibility of $x$ in $H^2(\hskn,\IZ)$, i.e.\ the positive generator of the ideal $(x, H^2(\hskn, \IZ)) \subset \IZ$. The transcendental lattice of $\hskn$ is $\trans(\hskn) = \ns(\hskn)^\perp \subset H^2(\hskn, \IZ)$.

\begin{prop} \label{prop: transc}
Let $S$ be an algebraic K3 surface with $\pic(S) = \IZ H$, $H^2 = 2t$, $t \geq 1$. For $n \geq 2$, let $\sigma \in \bir(\hskn)$ be a non-natural automorphism and $\nu  = bh - a\delta \in \ns(\hskn)$ be the generator of the line fixed by $\sigma^*$ with $a,b > 0$, $(a,b) = 1$. Then one of the following holds:
\begin{itemize}
\item $\sigma^*\vert_{\trans(\hskn)} = -\id$ and
\begin{itemize}
\item either $\overline{\sigma^*} = -\id$ and $\nu^2 = 2$, i.e.\ $(a,b)$ is an integer solution of $(n-1)X^2 - tY^2 = -1$;
\item or $\overline{\sigma^*} = \id$, $\nu^2 = 2(n-1)$ and $\divi(\nu)=n-1$, i.e.\ $(a,\frac{b}{n-1})$ is an integer solution of $X^2 - t(n-1)Y^2 = -1$.
\end{itemize}
\item $n \geq 9$, $\sigma^*\vert_{\trans(\hskn)} = \id$, $\overline{\sigma^*} = -\id$ and $\nu^2 = 2t$, i.e.\ the minimal solution of $P_{t(n-1)}(1)$  is $(b, \frac{a}{t})$ and $b \not \equiv \pm 1 \pmod{n-1}$.
\end{itemize}
\end{prop}

\begin{proof}
Let $(z,w)$ be the minimal solution of $P_{t(n-1)}(1)$ with $z \equiv \pm 1 \pmod{n-1}$. Then, as recalled in Proposition \ref{prop: action}, $z \equiv \pm 1 \pmod{2(n-1)}$ and $w \equiv 0 \pmod{2}$. By \cite[Theorem 1.3]{markman} $\sigma^*$ is a monodromy operator hence it acts on $A_{H^2(\hskn,\IZ)}$ as $\pm \id$ by \cite[Lemma 9.2]{markman}. Moreover, since $\trans(\hskn) \cong \trans(S)$ has odd rank, $\sigma^*\vert_{\trans(\hskn)} = \pm \id$ (because $\sigma^*$ is a Hodge isometry; see \cite[Corollary 3.3.5]{huybrechts}). If $t \neq 1$, by imposing that $\sigma^*\vert_{\trans(\hskn)}$ glues with the isometry $\sigma^*\vert_{\ns(\hskn)}$ of the form \eqref{eq: action}, we conclude (by \cite[Corollary 1.5.2]{nikulin}) that $\sigma^*\vert_{\trans(\hskn)} = \id$ if and only if $z \equiv 1 \pmod{2t}$ and $\sigma^*\vert_{\trans(\hskn)} = -\id$ if and only if $z \equiv -1 \pmod{2t}$. On the other hand $\overline{\sigma^*} = \id$ if and only if $z \equiv -1 \pmod{2(n-1)}$ and $\overline{\sigma^*} = -\id$ if and only if $z \equiv 1 \pmod{2(n-1)}$, for all $t \geq 1$ (see \cite[Remark 5.2]{catt_autom_hilb}).

If $\sigma^*$ acts as $-\id$ on $\trans(\hskn)$ and $t \geq 2$, then the statement follows as in the proof of \cite[Proposition 5.1]{catt_autom_hilb} and by \cite[Lemma 6.3]{catt_autom_hilb}. On the other hand, assume that $z \equiv 1 \pmod{2t}$ (hence, $\sigma^*\vert_{\trans(\hskn)} = \id$ or $t = 1$).
By Lemma \ref{lemma: congruence} and the fact that $(n-1)X^2 - tY^2 = 1$ has no integer solutions (Proposition \ref{prop: action})  we have $z \equiv 1 \pmod{2(n-1)}$, i.e.\ $\sigma^*$ acts as $-\id$ on the discriminant group of $H^2(\hskn,\IZ)$. Moreover $(z,w) = (2t(n-1)s^2+1,2us)$, where $(u,s)$ is the minimal solution of $P_{t(n-1)}(1)$. The axis of the reflection $\sigma^*\vert_{\ns(\hskn)}$ is spanned by $$(n-1)wh - (z-1)\delta = 2s(n-1)\left( uh - ts\delta \right).$$
As $\gcd(u, ts) = 1$ we conclude that $\nu = uh - ts\delta$, whose square is $2t$. The minimal solution $(u,s)$ of $P_{t(n-1)}(1)$ can have $u \not \equiv \pm 1 \pmod{n-1}$ only if $n \geq 9$. 
\end{proof}

\begin{rem}\label{rem: sympl}
Let $\omega$ be a generator of $H^{2,0}(\hskn)$. An automorphism $\sigma \in \bir(\hskn)$ is symplectic (i.e.\ $\sigma^* \omega = \omega$) if and only if $\sigma^*\vert_{\trans(\hskn)} = \id$. As a consequence of \cite[Lemma 3.5, Corollary 5.1]{mongardi_towards} (see also \cite[Proposition 4.1]{catt_autom_hilb}) any symplectic non-trivial birational automorphism of $\hskn$ is not biregular. We will return to this point in Section \ref{subs: decomposition}.
\end{rem}

\subsection{Classification of birational automorphisms}

Let $L$ be an even lattice and $l \in L$ such that $(l,l) \neq 0$. We define the reflection
\begin{equation*}
R_{l}: L_\IR \rightarrow L_\IR, \quad m \mapsto m - 2\frac{(m,l)}{(l,l)}l.
\end{equation*}
This map restricts to an isometry $R_l \in O(L)$ if and only if the divisibility of $l$ is either $\left|(l,l)\right|$ or $\left|(l,l)\right|/2$. The \emph{real spinor norm} of $L$ is the group homomorphism $\spinn^L: O(L) \rightarrow \IR^*/\left(\IR^*\right)^2 \cong \left\{ \pm 1\right\}$ defined as
 \[ \spinn^L(g) = \left( -\frac{(v_1, v_1)}{2}\right) \ldots \left( -\frac{(v_r, v_r)}{2}\right) \pmod{\left(\IR^*\right)^2}\]
 \noindent where $g_\IR = R_{v_1} \circ \ldots \circ R_{v_r}$ is the factorization of $g_\IR \in O(L_\IR)$ with respect to reflections defined by elements $v_i \in L$. In particular, if the signature of $L$ is $(l_+,l_-)$, after diagonalizing the bilinear form of $L$ over $\IR$ it is immediate to check that $\spinn^L(-\id) = (-1)^{l_+}$.

\begin{proof}[Proof of Theorem \ref{thm: iff}]
We start by assuming $t \geq 2$. By Proposition \ref{prop: action}, Proposition \ref{prop: transc} and Lemma \ref{lemma: congruence} the numerical conditions in the statement of the theorem are necessary for the existence of a birational automorphism of $\hskn$ with non-trivial action on $\ns(\hskn)$ (i.e.\ a non-trivial automorphism, since $t \neq 1$). Assume now that these conditions hold. If $n \neq 2$ the equation $(n-1)X^2 - tY^2 = 1$ has no integer solutions, since $(j,k) \neq (-1,1)$ (Lemma \ref{lemma: congruence}). We therefore have $\overline{\movb(\hskn)} = \langle h, zh - tw\delta\rangle_{\IR_{\geq 0}}$ by \cite[Proposition 13.1]{bayer_macri_mmp}. Using again Lemma \ref{lemma: congruence}, we give the following definitions depending on $(j,k)$.
\begin{itemize}
\item[($i$)] If $(j,k) = (1,-1)$, then $(n-1)a^2 - tb^2 = -1$ has integer solutions. Let $(a,b)$ be the solution with smallest $a > 0$ and set $\nu = bh - a\delta$. By Lemma \ref{lemma: congruence} we have $(z,w) = (2(n-1)a^2 + 1, 2ab)$, since we are assuming $t \geq 2$. As $\nu^2 = 2$ we can define $\phi = -R_\nu \in O(H^2(\hskn,\IZ))$. Then $\overline{\phi} = -\id \in O(A_{H^2(\hskn,\IZ)})$ by \cite[Proposition 3.1]{ghs_kodaira}.
\item[($ii$)] If $(j,k) = (-1,-1)$, then $P_{t(n-1)}(-1)$ is solvable. Let $(a,b)$ be the minimal solution and set $\nu = (n-1)bh - a\delta$. Then $(z,w) = (2t(n-1)b^2 - 1, 2ab)$. The class $\nu$ has square $2(n-1)$ and divisibility $n-1$ (see \cite[Lemma 6.3]{catt_autom_hilb}), so we can define $\phi = -R_\nu \in O(H^2(\hskn,\IZ))$ and $\overline{\phi} = \id$ by generalizing \cite[Corollary 3.4]{ghs_kodaira}.
\item[($iii$)] If $(j,k) = (1,1)$, let $(b,a)$ be the minimal solution of $P_{t(n-1)}(1)$ and set $\nu = bh - ta\delta$. Then $(z,w) = (2t(n-1)a^2+1,2ab)$. The primitive element $\gamma = (n-1)ah - b\delta$ satisfies $\gamma^2 = -2(n-1)$, while the divisibility of $\gamma$ in $H^2(\hskn,\IZ)$ is either $n-1$ or $2(n-1)$. This implies that $\phi = R_\gamma \in O(H^2(\hskn,\IZ))$ and moreover $\overline{\phi} = -\id$ by \cite[Proposition 9.12]{markman}.
\end{itemize}

By construction, $\phi\vert_{\trans(\hskn)} = - \id$ in cases ($i$), ($ii$) and $\phi\vert_{\trans(\hskn)} = \id$ in case ($iii$).
Thus $\phi$ extends to a Hodge isometry of $H^2(\hskn, \IC)$. Moreover $\phi$ is orientation-preserving. Indeed, let $\spinn \coloneqq \spinn^{H^2(\hskn,\IZ)}$; then if $\phi = -R_\nu$ we have $$\spinn(\phi) = \spinn(-\id)\spinn(R_\nu) = - \spinn(R_\nu) = \textrm{sgn}(\nu^2) = +1$$
\noindent where we used the fact that $H^2(\hskn,\IZ)$ has signature $(3,20)$. If $\phi = R_\gamma$, then $\spinn(\phi) = -\textrm{sgn}(\gamma^2) = +1$. We conclude, by \cite[Lemma 9.2]{markman}, that $\phi \in \Mo^2_{\textrm{Hdg}}(\hskn)$ (in the case $\phi = R_\gamma$, also see \cite[Proposition 9.12]{markman}). From the relations between $(a,b)$ and $(z,w)$ that we remarked for each of the three pairs $(j,k)$, it is immediate to check that $\nu \in \movb(\hskn)$ and that $\phi\vert_{\ns(\hskn)}$ is the isometry \eqref{eq: action} fixing the line spanned by $\nu$, hence $\phi(\movb(\hskn)) = \movb(\hskn)$. Then by \cite[Lemma 6.22 and Lemma 6.23]{markman} there exists $\sigma \in \bir(\hskn)$ such that $\sigma^* = \phi$.

We are left with the case $t = 1$. For all $n \geq 2$ the group $\aut(\hskn)$ contains the natural involution induced by the generator of $\aut(S)$. 
Let $(z,w)$ be the minimal solution of $P_{n-1}(1)$ with $z \equiv \pm 1 \pmod{n-1}$. If there exists a non-natural automorphism then $w \equiv 0 \pmod{2}$, which implies $z \equiv 1 \pmod{2t}$. Thus, by the proof of Proposition \ref{prop: transc}, $(z,w)$ is not the minimal solution of the equation, i.e.\ $b \not \equiv \pm 1 \pmod{n-1}$. On the other hand, if we assume $b \not \equiv \pm 1 \pmod{n-1}$ then $(n-1)X^2 - Y^2 = 1$ has no integer solutions by \cite[Lemma A.2]{debarre}. The existence of a symplectic automorphism follows as in the proof for $t \geq 2$, case ($iii$). By \cite[Proposition 1.1]{catt_autom_hilb} this automorphism is not biregular.
\end{proof}


\begin{rem}
Notice that $t=1$ is the only value of $t$ such that the isometry \eqref{eq: action} of $\ns(\hskn)$ can be glued to both $+\id, -\id \in O(\trans(\hskn))$, by \cite[Corollary 1.5.2]{nikulin}. If $t=1$ let $\nu = bh - a\delta \in \ns(\hskn)$ be as in the proof of Theorem \ref{thm: iff}. Then $(-R_\nu) \oplus (-\id) \in O(\ns(\hskn) \oplus \trans(\hskn))$ extends to $-R_\nu \in O(H^2(\hskn,\IZ))$, while $(-R_\nu) \oplus \id$ extends to $R_\gamma$, with $\gamma = (n-1)ah-b\delta$. If $n-1$ is not a square and $b \not \equiv \pm 1 \pmod{n-1}$ both isometries $R_\gamma$ and $-R_\nu$ of $H^2(\hskn,\IZ)$ are Hodge monodromies: they lift to the two non-natural birational automorphisms of $\hskn$. 
\end{rem}

\begin{rem}
If for $n \geq 2$, $t \geq 2$ there exists a non-trivial birational automorphism on $\hskn$, it is biregular if and only if $n, t$ satisfy condition ($iii$) of \cite[Theorem 6.4]{catt_autom_hilb}, which guarantees that $\mathcal{A}_{\hskn} = \movb(\hskn)$ (see also Theorem \ref{thm: ambig}). In particular, if $t \leq 2n - 3$ the automorphism is not biregular. 
\end{rem}

The following proposition provides examples of polarization degrees for the K3 surface $S$ so that $\hskn$ has a non-natural birational automorphism, for each of the three different actions on cohomology of Proposition \ref{prop: transc}. 

\begin{prop}\label{prop:families}
Fix $n \geq 2$. For $t \geq 1$, we denote by $S$ an algebraic K3 surface with $\pic(S) = \IZ H$, $H^2 = 2t$.
\begin{enumerate}
\item[($i$)] There exist infinitely many $t$'s such that $\bir(\hskn)$ contains a non-symplectic automorphism $\sigma$ with $H^2(\hskn,\IZ)^{\sigma^*} \cong \langle 2 \rangle$, e.g.\ $t = (n-1)k^2 + 1$ for $k \geq 1$.
\item[($ii$)] There exists $t$ such that $\bir(\hskn)$ contains a  non-symplectic automorphism $\sigma$ with $H^2(\hskn,\IZ)^{\sigma^*} \cong \langle 2(n-1) \rangle$ if and only if $-1$ is a quadratic residue modulo $n-1$. If so, this happens for infinitely many $t$'s, e.g.\ $t = (n-1)k^2 + 2qk + \frac{q^2+1}{n-1}$ for $k, q \geq 1$ and $q^2 \equiv -1 \pmod{n-1}$.
\item[($iii$)] There exists $t$ such that $\bir(\hskn)$ contains a non-trivial symplectic automorphism $\sigma$ if and only if $n-1 = \frac{q^2 - 1}{h}$ for some $q \geq 3, h \not \equiv 0 \pmod{q\pm1}$. If so, $H^2(\hskn,\IZ)_{\sigma^*} \cong \langle -2(n-1) \rangle$ and this happens for infinitely many $t$'s, e.g.\ $t = (n-1)k^2 + 2qk + h$ for $k \geq 1$.
\end{enumerate} 
\end{prop}

\begin{proof}
As before, we denote by $(z,w)$ the minimal solution of $P_{t(n-1)}(1)$ with $z \equiv \pm 1 \pmod{n-1}$.
\begin{enumerate}
\item[($i$)] If $t = (n-1)k^2 + 1$ and $k \geq 1$, then $t(n-1)$ is not a square and $(z,w) = (2t-1,2k) = (2(n-1)k^2 + 1, 2k)$ by \cite[Lemma A.2]{debarre}, since $(k,1)$ is the solution of $(n-1)X^2 - tY^2 = -1$ with smallest $X > 0$. We conclude with Theorem \ref{thm: iff}.
\item[($ii$)] If $\bir(\hskn)$ contains a non-symplectic automorphism $\sigma$ with $H^2(\hskn,\IZ)^{\sigma^*} \cong \langle 2(n-1) \rangle$, then by Theorem \ref{thm: iff} and Lemma \ref{lemma: congruence} the equation $P_{t(n-1)}(-1)$ is solvable, hence $-1$ is a quadratic residue modulo $n-1$. On the other hand, let $t = (n-1)k^2 + 2qk + \frac{q^2+1}{n-1}$ for $k, q \geq 1$ and $q^2 \equiv -1 \pmod{n-1}$. The minimal solution of $P_{t(n-1)}(-1)$ is $(q+(n-1)k,1)$, therefore $t(n-1)$ is not a square and $(z,w) = (-1+2t(n-1), 2q + 2(n-1)k)$, again by \cite[Lemma A.2]{debarre}. The existence of the automorphism and its action on cohomology follow then from Theorem \ref{thm: iff}.
\item[($iii$)] If $\bir(\hskn)$ contains a symplectic automorphism $\sigma \neq \id$, then by Proposition \ref{prop: action} the minimal solution $(u,s)$ of $P_{t(n-1)}(1)$ has $u \not \equiv \pm 1 \pmod{n-1}$. Hence $n-1 = \frac{u^2 - 1}{ts^2} \geq 8$ and $u \pm 1 \nmid ts^2$. Let now $t = (n-1)k^2 + 2qk + h$ for $k,q,h$ as in the statement. The minimal solution of $P_{t(n-1)}(1)$ is $(u,s) = (q+(n-1)k,1)$. Since $h \not \equiv 0 \pmod{q \pm 1}$, we have $q \not \equiv \pm 1 \pmod{n-1}$, therefore $(z,w) = (1+2t(n-1), 2q + 2(n-1)k)$, which allows us to conclude by Theorem \ref{thm: iff}. \qedhere
\end{enumerate}
\end{proof}

By \cite[Proposition 6.7]{catt_autom_hilb}, if $t = (n-1)k^2 + 1$ as in case ($i$) of Proposition \ref{prop:families} the automorphism of $\hskn$ is biregular whenever $k \geq \frac{n+3}{2}$. The argument in the proof of loc.\ cit.\ can be easily adapted to show the biregularity also for the automorphism of the family presented in case ($ii$), if $k \geq \frac{n+3}{2}$.



\section{Decomposition of the movable cone and automorphisms of birational models} \label{subs: decomposition}

We consider the wall-and-chamber decomposition \eqref{eq: wall-and-ch} of the closed movable cone $\overline{\movb(\hskn)} \subset \ns(\hskn)_{\IR} = \IR h \oplus \IR\delta$. As shown in \cite[Lemma 2.5]{catt_autom_hilb} by using \cite[Theorem 12.1]{bayer_macri_mmp}, the walls are spanned by the classes of the form $Xh - 2tY\delta$ which lie in the movable cone, for $(X, Y)$ a positive solution of one of Pell's equations $X^2 - 4t(n-1)Y^2 = \alpha^2 - 4\rho(n-1)$ with $X \equiv \pm \alpha \pmod{2(n-1)}$, where the possible values of $\rho$ and $\alpha$ are:
\begin{align} \label{eq: alpha_rho}
\begin{cases}
\rho = -1 \text{ and } 1 \leq \alpha \leq n-1; \\
\rho = 0 \text{ and } 3 \leq \alpha \leq n-1;\\
1 \leq \rho < \frac{n-1}{4} \text{ and } 4\rho+1 \leq \alpha \leq n-1.
\end{cases}
\end{align}

\begin{defi}
For $n \geq 2$, let $S$ be an algebraic K3 surface with $\pic(S) = \IZ H$, $H^2 =2t$, $t \geq 1$. The value $t$ is said to be \emph{$n$-irregular} if there exists $\sigma \in \bir(\hskn)$ such that for all ihs birational models $g: \hskn \dasharrow X$ the birational automorphism $g \circ \sigma \circ g^{-1} \in \bir(X)$ is not biregular.
\end{defi}

Notice that $\sigma^*\vert_{\ns(\hskn)}$ is a reflection which, by \cite[Lemma 5.12]{markman}, acts on the set of K\"ahler-type chambers of $\overline{\movb(\hskn)}$. If one chamber is preserved, then by \cite[Theorem 1.3]{markman} the automorphism becomes biregular on the corresponding ihs birational model of $\hskn$. Hence $t$ is $n$-irregular if and only if $\bir(\hskn)$ contains a non-natural birational automorphism and its action on $\ns(\hskn)$ fixes one of the walls in $\movb(\hskn)$ (i.e.\ the number of chambers in the decomposition is even).


\begin{prop} \label{prop: t=1 e id}
For $n \geq 2$ and $t \geq 1$, assume that there exists a non-natural birational automorphism $\sigma \in \bir(\hskn)$ and either $t=1$ or $\sigma^*\vert_{\trans(\hskn)} = \id$. Then $t$ is $n$-irregular. 
\end{prop} 

\begin{proof}
By the proof of Proposition \ref{prop: transc}, if $t = 1$ or $\sigma^*\vert_{\trans(\hskn)} = \id$ the axis of the reflection $\sigma^*\vert_{\ns(\hskn)}$ is spanned by the class $\nu = bh - ta\delta$, where $(b,a)$ is the minimal solution of $P_{t(n-1)}(1)$. Moreover $n \geq 9$ and $\overline{\movb(\hskn)} = \langle h, zh - tw\delta \rangle_{\IR_{\geq 0}}$ with $(z,w) = (2t(n-1)a^2+1,2ab)$. Since $(b+1)(b-1) = t(n-1)a^2$, we have $c \coloneqq \max\{\gcd(n-1, b-1), \gcd(n-1,b+1)\} \geq 2$. We define $\alpha \coloneqq \max\{4, \frac{2(n-1)}{c}\}$, which is an even integer such that $4 \leq \alpha \leq n-1$. We consider Pell's equation $P_{4t(n-1)}(\alpha^2)$. The pair $(X,Y) = (b \alpha, a\frac{\alpha}{2})$ is a solution with the property $\frac{2Y}{X} = \frac{a}{b}$. We also have $X = b\alpha \equiv \pm \alpha \pmod{2(n-1)}$ by construction, hence $\nu$ lies on the wall (in the interior of the movable cone) spanned by the class $Xh - 2tY\delta$.
\end{proof}

By Proposition \ref{prop:families}, for $n$ fixed the number of $n$-irregular values $t$ as in Proposition \ref{prop: t=1 e id} is either zero or infinite.
For non-symplectic automorphisms when $t \neq 1$, the behaviour is different: for a fixed $n \geq 2$ there is only a finite number of $n$-irregular values $t$ for which the non-natural birational automorphism of $\hskn$ acts as $-\id$ on $\trans(\hskn)$. We provide an algorithm to compute them. 

\begin{prop} \label{prop: ambiguous}
Let $t \geq 2$ and $n \geq 2$ such that $t(n-1)$ is not a square and $(n-1)X^2 - tY^2 = 1$ has no integer solutions (if $n \neq 2$). Assume that one between $(n-1)a^2 - tb^2 = -1$ and $a^2 - t(n-1)b^2 = -1$ admits integer solutions; define $\ell=1$ if the first equation is solvable or $\ell=n-1$ if the second one is, and let $(a,b)$ be the positive solution with smallest $a > 0$. Then $t$ is $n$-irregular if and only if there exists a pair $(\alpha, \rho)$ as in \eqref{eq: alpha_rho} and a positive solution $(X,Y)$ of $P_{4t(n-1)}(\alpha^2 - 4\rho(n-1))$ with $X \equiv \pm \alpha \pmod{2(n-1)}$ such that $4t\ell Y^2 = (\alpha^2-4\rho(n-1))a^2$.
\end{prop}

\begin{proof}
Let $\nu = \ell bh - a\delta \in \movb(\hskn)$ be the class of square $2\ell$ as in the proof of Theorem \ref{thm: iff}. Then $\nu$ lies on one of the walls in the interior of $\movb(\hskn)$ if and only if $\frac{a}{\ell b} = 2t\frac{Y}{X}$, where $(X,Y)$ is a positive solution of one of the equations $P_{4t(n-1)}(\alpha^2 - 4\rho(n-1))$ with $X \equiv \pm \alpha \pmod{2(n-1)}$ and $(\alpha, \rho)$ as in \eqref{eq: alpha_rho}. We observe that $(2(n-1)a)^2-4t(n-1)(\ell b)^2 = -4\ell (n-1)$, hence the condition on the slopes becomes
\[\sqrt{\frac{1}{4t(n-1)}-\frac{\alpha^2-4\rho(n-1)}{4t(n-1)X^2}} \sqrt{\frac{1}{4t(n-1)}+\frac{4\ell (n-1)}{4t(n-1)(2(n-1)a)^2}} = \frac{1}{4t(n-1)}. \]
If we rearrange the equation we obtain
\begin{equation} \label{eq: kX^2-ca^2}
\ell X^2 - (n-1)(\alpha^2-4\rho(n-1))a^2 = \ell (\alpha^2-4\rho(n-1))
\end{equation}
\noindent i.e.\ $4t\ell Y^2 = (\alpha^2-4\rho(n-1))a^2$.
\end{proof} 

\begin{rem} \label{rem: irregulars}
\begin{itemize}
\item Notice that, by \eqref{eq: kX^2-ca^2} and $(n-1)a^2 - t(\ell b)^2 = -\ell $, we have $X^2 = (\alpha^2 - 4\rho(n-1))t\ell b^2$. However, $t\ell $ has to divide $\alpha^2 - 4\rho(n-1)$ (because $t\ell $ is coprime with $a$), hence $\alpha^2 - 4\rho(n-1) = t\ell r^2$ for some $r \in \IN$. This gives an easy way to compute a (finite) list of candidates for the $n$-irregular values $t$, among the divisors of $\alpha^2 - 4\rho(n-1)$ for $(\alpha, \rho)$ as in \eqref{eq: alpha_rho}. In particular, $t \leq (n-1)(n+3)$.
\item In order to check if a value $t \geq 2$ is $n$-irregular in Proposition \ref{prop: ambiguous}, it is enough to consider the pairs $(\alpha, \rho)$ such that $\alpha^2 - 4\rho(n-1) > 0$ and the positive solutions $(X,Y)$ with smallest $X$ in each equivalence class of solutions of $P_{4t(n-1)}(\alpha^2 - 4\rho(n-1))$, otherwise the wall spanned by $Xh - 2tY\delta$ is not in the interior of the movable cone.  
\end{itemize}
\end{rem}

If we combine Proposition \ref{prop: ambiguous} with Proposition \ref{prop: transc} and Theorem \ref{thm: iff} we get the following.

\begin{cor}\label{cor: finite}
For a fixed $n \geq 2$, the number of $n$-irregular values $t$ for which $\bir(\hskn)$ contains a non-symplectic birational automorphism is finite. In particular, for $n \leq 8$ there is only a finite number of $n$-irregular values $t$.
\end{cor}


In Table \ref{table: ambig}, for $n \leq 14$ we list all $n$-irregular values $t$ as in Corollary \ref{cor: finite}. We separate the values of $t$ for which the middle wall of the movable cone is spanned by a primitive class $\nu$ of square $2$ (i.e.\ $\ell =1$), from those where the generator has square $2(n-1)$ and divisibility $n-1$ (i.e.\ $\ell =n-1$).
For $n \neq 9,12$ the values $t$ in the table are all the $n$-irregular values.
We observe that, for $n \leq 5$, the $n$-irregular values are all of the form $t=n$ or $t=4n-3$. 

\begin{table}[ht] 
\caption{$n$-irregular values $t$ for $n \leq 14$ as in Corollary \ref{cor: finite}.}
\begin{tabular}{|c|c|c|}
\hline
$n$ & $n$-irregular $t$'s s.t.\ $\nu^2 = 2$ & $n$-irregular $t$'s s.t.\ $\nu^2 = 2(n-1)$\\
\hline
$2$ & \multicolumn{2}{c|}{$5$}\\
\hline
$3$ & $3, 9$ & / \\
\hline 
$4$ & $4, 13$ & / \\
\hline
$5$ & $5,17$ & / \\
\hline
$6$ & $6,9,21$ & / \\
\hline 
$7$ & $7,25,49$ & / \\
\hline
$8$ & $2,4,8,11,16,29,37$ & / \\
\hline
$9$ & $1,9,33,57$ & / \\
\hline
$10$ & $10,13,37,61,85$ & / \\
\hline
$11$ & $11,19,41,49,121$ & / \\
\hline
$12$ & $3,4,5,12,15,25,27,45,125$  & / \\
\hline
$13$ & $1, 13, 49$ & / \\
\hline
$14$ & $14,17,22,38,49,53,77,121,133$ & $5$ \\
\hline
\end{tabular}
\label{table: ambig}
\end{table}

%
%
%

\begin{lemma}\label{lemma: t=n,4n-3}
If $n = tr^2 \geq 3$ with $r \in \IN$ and $(t, n) \neq (1, 4)$, the value $t$ is $n$-irregular.
For all $n \geq 2$ and $r \in \IN$ such that $t = \frac{4n - 3}{r^2} \in \IN$, the value $t$ is $n$-irregular unless $(t, n) = (1, 3), (1,7)$.
\end{lemma}

\begin{proof}
\begin{itemize}
\item Let $t = \frac{n}{r^2} \in \IN$. By Theorem \ref{thm: iff}, if $t = 1$ and $r \neq 2$ we have $\bir(\hskn) \cong \IZ/2 \IZ \times \IZ/2 \IZ$ and the statement follows from Proposition \ref{prop: t=1 e id}. If $t \geq 2$, the positive solution of $(n-1)a^2 - t b^2= -1$ with smallest $a$ is $(a,b)=(1,r)$, hence by \cite[Lemma A.2]{debarre} the minimal solution of $X^2 - t(n-1)Y^2 = 1$ is $(z,w)=(2tr^2 - 1, 2r) = (2(n-1) + 1, 2r)$. Thus the Hilbert scheme $\hskn$ has a non-symplectic birational automorphism by Theorem \ref{thm: iff}. Proposition \ref{prop: ambiguous} allows us to conclude by noticing that $(X,Y)=(2tr^2, r)$ is a solution for the equation $P_{4t(n-1)}(\alpha^2 - 4\rho(n-1))$ with $\rho = -1, \alpha = 2$, satisfying $X \equiv \alpha \pmod{2(n-1)}$ and $4t Y^2 = (\alpha^2-4\rho(n-1))a^2$. 

\item If $t = \frac{4n-3}{r^2} = 1$, there exists a non-natural birational automorphism on $\hskn$ if and only if $n \neq 3, 7$. If $t \geq 2$, we can assume that the minimal solution of $(n-1)a^2 - tb^2 = -1$ has $a \neq 1$ (otherwise we are in the previous case $n = tb^2$), hence it is $(a,b) = (2,r)$. The pair $(X,Y) = (4n-3,r)$ is a solution of $P_{4t(n-1)}(\alpha^2 - 4\rho(n-1))$ with $\rho = -1$, $\alpha = 1$ and it satisfies $X \equiv \alpha \pmod{2(n-1)}$, so the relation $4tY^2 = (\alpha^2 - 4\rho(n-1))a^2$ of Proposition \ref{prop: ambiguous} holds. \qedhere
\end{itemize}


\end{proof}

A special case of Lemma \ref{lemma: t=n,4n-3} is $t = n$, where it is known that the automorphism of $\hskn$ which acts as the reflection in the (only) wall contained in the interior of $\movb(\hskn)$ is Beauville's involution \cite[\S 6]{beauville_rmks}, which is biregular if and only if $n=2$.

%
%

\begin{rem}
We provide a modular interpretation of $n$-irregular values. Let $\mathcal{M}^n_{2,\gamma}$ be the moduli space of $2$-polarized ihs manifolds of $K3^{[n]}$-type with polarization of divisibility $\gamma\in \lbrace 1,2 \rbrace$. By \cite[Proposition 3.2]{apostolov}, $\mathcal{M}^n_{2,\gamma}$ is connected if it is not empty (which is always the case when $\gamma = 1$, while for $\gamma = 2$ we need $n \equiv 0 \pmod{4}$). We have a period map $\mathcal{M}^n_{2,\gamma} \ra \mathcal{P}^n_{2,\gamma}$ defined as in \cite[\S 3.9]{debarre}, which is no longer surjective. In the case $n=2$ its image was computed in \cite[\S 3]{debarre_macri}. For $s \in \IN$, let $\mathcal{D}_{2s}$ be the Heegner divisor in $\mathcal{P}^n_{2,\gamma}$ whose preimage in $\mathcal{M}^n_{2,\gamma}$ is the Noether-Lefschetz divisor of manifolds which are special of discriminant $2s$ (we refer to \cite[\S 3]{debarre} for notation and details). For a $2t$-polarized K3 surface $S$ as in Theorem \ref{thm: iff} with $\sigma \in \bir(\hskn)$ non-natural and $H^2(\hskn, \IZ)^{\sigma^*} = \IZ \nu \cong \langle 2 \rangle$, the divisibility of $\nu$ is $\gamma = \gcd(b,2(n-1)) = \gcd(b,2) \in \left\{ 1, 2\right\}$, where $(a,b)$ is the positive solution of $(n-1)a^2 - t b^2=-1$ with smallest $a$. If we assume that $t$ is $n$-irregular, for any ihs birational model $g: \hskn \dasharrow X$ the big and nef divisor $D$ on $X$ such that $g^*(D) = \nu$ is not ample. Thus there is an irreducible component of $\mathcal{D}_{2t}\subset \mathcal{P}^n_{2,\gamma}$ which lies outside of the image of the period map. Proposition \ref{prop: ambiguous} and Lemma \ref{lemma: t=n,4n-3} can then be used to produce examples of such divisors. For instance, for every $n\geq 3$ both $\mathcal{D}_{2n}$, $\mathcal{D}_{8n-6}\subset \mathcal{P}^n_{2,1}$ have at least one irreducible component outside the image of the period map.  
\end{rem}

\section{K\"ahler-type chambers and automorphisms for $n=3$} \label{sec: n=3}

It is known that for a K3 surface $S$ of Picard rank one the number of K\"ahler-type chambers in the decomposition of $\overline{\movb(S^{[2]}}$) is $d \in \left\{ 1,2,3 \right\}$ (see for instance \cite[Example 3.18]{debarre}). We now detail the computation of the number of chambers for $n=3$ in the cases where $\overline{\movb(\hskn)} = \langle h, zh-tw\delta\rangle_{\IR_{\geq 0}}$, with $(z,w)$ the minimal solution of $P_{t(n-1)}(1)$ with $z \equiv \pm 1 \pmod{n-1}$. In particular, this holds whenever $\bir(\hskt) \neq \left\{ \id \right\}$.



Let $S$ be an algebraic K3 surface such that $\pic(S) = \IZ H$, $H^2 = 2t$, $t \geq 2$. As explained in Section \ref{subs: decomposition}, if $\bir(\hskt) \neq \left\{ \id \right\}$ then $2t$ is not a square, $2X^2 - tY^2 = 1$ has no integer solutions and the minimal solution of $P_{2t}(1)$ has $Y \equiv 0 \pmod{2}$.
As a consequence $\overline{\movb(\hskt)} = \langle h, zh-2tw\delta\rangle_{\IR_{\geq 0}}$, where $(z,w)$ is the minimal solution of $\pone$. The walls in the interior of the movable cone are the rays through $Xh - 2tY\delta$, for $(X,Y)$ positive solution of $P_{8t}(8 + \alpha^2)$ with $\alpha \in \left\{1,2 \right\}$ and $0 < \frac{Y}{X} < \frac{w}{z}$ (all solutions satisfy $X \equiv \pm \alpha \pmod{4}$). Clearly if $(X,Y)$ is a solution of $\pnine$ and $(X', Y')$ is a solution of $\ptwelve$, then $\frac{Y}{X} \neq \frac{Y'}{X'}$. Thus the number of chambers in the movable cone coincides with the number of combined equivalence classes of solutions for $\pnine$ and $\ptwelve$ (the class of the solution $(3,0)$ of $\pnine$ determines the two extremal rays of the movable cone). This, combined with Lemma \ref{lemma: pnine} and Lemma \ref{lemma: ptwelve}, gives the following result.

\begin{prop}\label{prop: numb_chambers}
Let $t \geq 2$ such that $2t$ is not a square, $2X^2 - tY^2 = 1$ has no integer solutions and the minimal solution of $X^2 - 2tY^2 = 1$ has $Y \equiv 0 \pmod{2}$. The possible numbers of chambers in the movable cone of $S^{[3]}$, for $S$ an algebraic K3 surface such that $\pic(S) = \IZ H$, $H^2 = 2t$, are as follows:

\begin{center}
\begin{tabular}{c|c|c|c|c|c|c|c|c|c}
$t \mod 18$ & $0$ & $1$ & $2$ & $3$ & $4$ & $5$ & $6$ & $7$ & $8$\\ \hline
\emph{\# chambers} & $1,2,3$ & $1$ & $1,3$ & $1,2$ & $1$ & $1,3,5$ & $1$ & $1$ & $1,3$
\end{tabular}

\medskip
\begin{tabular}{c|c|c|c|c|c|c|c|c|c}
$t \mod 18$ & $9$ & $10$ & $11$ & $12$ & $13$ & $14$ & $15$ & $16$ & $17$\\ \hline
\emph{\# chambers} & $1,2,3$ & $1$ & $1,3,5$ & $1$ & $1$ & $1,3$ & $1$ & $1$ & $1,3,5$
\end{tabular}
\end{center} 

In particular, if $t \equiv 1,4,6,7,10,12,13,15,16 \pmod{18}$, then $\bir(S^{[3]}) = \aut(S^{[3]})$.
\end{prop}


From Theorem \ref{thm: iff}, Corollary \ref{cor: finite} and the results of this section we conclude the following.

\begin{prop} \label{prop: bir iff}
Let $S$ be an algebraic K3 surface such that $\pic(S) = \IZ H$, $H^2 = 2t$. If $t = 1$, then $\bir(\hskt) = \aut(\hskt) \cong \IZ/2\IZ$. If $t \geq 2$, then $\bir(\hskt) \neq \left\{ \id \right\}$ if and only if:
\begin{itemize}
\item $2t$ is not a square;
\item $2X^2 - tY^2 = 1$ has no integer solutions;
\item either $2X^2 - tY^2 = -1$ or $X^2 - 2tY^2 = -1$ has integer solutions.
\end{itemize}
If $\bir(\hskt) \neq \left\{ \id \right\}$, let $d$ be the number of chambers in the decomposition of $\overline{\movb(\hskt)}$. Then $d \in \left\{ 1,2,3,5 \right\}$ and one of the following holds:
\begin{itemize}
\item $d = 1$ and $\bir(S^{[3]}) = \aut(S^{[3]}) \cong \IZ/2\IZ$;
\item $d = 2$, $t = 3$ or $t = 9$, $\aut(\hskt) = \left\{ \id \right\}$ and $\bir(\hskt) \cong \IZ/2\IZ$;
\item $d=3,5$, $\aut(\hskt) = \left\{ \id \right\}$ and $\bir(S^{[3]}) \cong \IZ/2\IZ$
\end{itemize}
If $t \neq 3,9$ and $\sigma \in \bir(\hskt)$, there exists an ihs sixfold $X$ and a birational map $g: \hskt \dashrightarrow X$ such that $g \circ \sigma \circ g^{-1} \in \aut(X)$. 
\end{prop}

%

\section{Ambiguous Hilbert schemes and birational models} \label{sec: ambig}

Having classified the group of birational automorphisms of $\hskn$, for a K3 surface $S$ with Picard rank one, we explain  in this section how the same approach can be used to study the more general problem of whether there exists a K3 surface $\Sigma$, again of Picard rank one, and a birational map $\phi: \hskn \dashrightarrow \sigman$ which do not come from an isomorphism $S \overset{\cong}{\ra} \Sigma$. This is related to the notion of \emph{(strong) ambiguity} for Hilbert schemes of points on K3 surfaces, which for $n=2$ (and $\phi$ biregular) has been investigated in \cite{debarre_macri} and \cite{zuffetti}. Some of the results of this section overlap (even though they are proved in a different way) with those of \cite{mmy}, where birationality of derived equivalent Hilbert schemes of 
K3 surfaces is studied.

If $\hskn$ and $\sigman$ are birational, then $S$ and $\Sigma$ are Fourier--Mukai partners and $\hskn$ and $\sigman$ are derived equivalent (see \cite[Proposition 10]{ploog}). If $\pic(S) = \IZ H$ with $H^2 = 2t$, $t \geq 1$, then by \cite[Proposition 1.10]{oguiso_primes} the number of non-isomorphic FM partners of $S$ is $2^{\rho(t)-1}$, where $\rho(t)$ denotes the number of prime divisors of $t$ (and $\rho(1) = 1$). We can classify the Fourier--Mukai partners $\Sigma$ of $S$ as follows (for details see for instance \cite[\S 4]{oguiso_primes}). The overlattice $L = H^2(\Sigma, \IZ)$ of $\trans(S) \oplus \ns(S)$ (with integral Hodge structure defined by setting $L^{2,0} = \trans(S)^{2,0}$) corresponds to an isotropic subgroup $I_L \subset A_{\trans(S)} \oplus A_{\ns(S)} = \frac{\IZ}{2t\IZ} (-\frac{1}{2t}) \oplus \frac{\IZ}{2t\IZ} (\frac{1}{2t})$ (as in \cite[\S 1.4]{nikulin}). Since $\ns(S)$ and $\trans(S)$ are primitive in $L$, the group $I_L$ is of the form $I_L = I_a \coloneqq \langle \epsilon + a\eta \rangle$ for some $a \in \left(\IZ/2t\IZ\right)^\times$, $a^2 \equiv 1 \pmod{4t}$, where $\epsilon$ (respectively, $\eta$) is a generator of $A_{\trans(S)}$ (respectively, $A_{\ns(S)}$) on which the finite quadratic form takes value $-\frac{1}{2t}$ (respectively, $+\frac{1}{2t}$). For each $a \in \left(\IZ/2t\IZ\right)^\times$, $a^2 \equiv 1 \pmod{4t}$, there exists a K3 surface $\Sigma_a$ (unique up to isomorphism) such that $H^2(\Sigma_a, \IZ)$ is Hodge-isometric to the overlattice $L_a$ of $\trans(S) \oplus \ns(S)$ defined by $I_a$. Moreover, $\Sigma_a \cong \Sigma_b$ if and only if $b \equiv \pm a \pmod{2t}$. Indeed, $2^{\rho(t)-1}$ is the cardinality of $\left\{ a \in \left(\IZ/2t\IZ\right)^\times, a^2 \equiv 1 \pmod{4t} \right\}/\pm \id$.

\begin{rem} \label{rem: unique model}
If there exists a birational map $\phi: \hskn \dashrightarrow \sigman$ not induced by an isomorphisms of the K3 surfaces, then $\Sigma$ is uniquely determined up to isomorphism as $\phi^*(\movb(\sigman)) = \movb(\hskn)$ and $\phi^*$ needs to invert the primitive generators of the extremal rays of the movable cones.
This, together with Theorem \ref{thm: ambig} below, implies that for any $\hskn$ there is at most one chamber of $\movb(\hskn)$ other than the ample cone for which the associated birational ihs model is an Hilbert scheme of points. 

\end{rem}

\begin{theorem}\label{thm: ambig}
Let $S$ be an algebraic K3 surface such that $\pic(S) = \IZ H$, with $H^2 = 2t$, $t \geq 1$. For $n \geq 2$, let $(z,w)$ be the minimal solution of $P_{t(n-1)}(1)$ with $z \equiv \pm 1 \pmod{n-1}$. There exists a K3 surface $\Sigma$ and a birational map $\phi: \hskn \dashrightarrow \sigman$ which is not induced by an isomorphism $S \ra \Sigma$ if and only if:
\begin{itemize}
\item $t(n-1)$ is not a square;
\item if $n \neq 2$, $(n-1)X^2 - tY^2 = 1$ has no integer solutions;
\item $z \equiv \pm 1 \pmod{2(n-1)}$ and $w \equiv 0 \pmod{2}$.
\end{itemize}

If so, the K3 surfaces $S$ and $\Sigma$ are isomorphic if and only if $z \equiv \pm 1 \pmod{2t}$. Moreover, $\phi$ is biregular if and only if, for all integers $\rho, \alpha$ as follows:
\begin{itemize}
\item $\rho = -1$ and $1 \leq \alpha \leq n-1$, or
\item $\rho = 0$ and $3 \leq \alpha \leq n-1$, or
\item $1 \leq \rho < \frac{n-1}{4}$ and $\max\left\{4\rho+1,\ceil[\bigg]{2\sqrt{\rho(n-1)}}  \right\} \leq \alpha \leq n-1$
\end{itemize}
if Pell's equation $P_{4t(n-1)}(\alpha^2 - 4\rho(n-1))$ is solvable, the minimal solution $(X,Y)$ with $X \equiv \pm \alpha \pmod{2(n-1)}$ satisfies $\frac{Y}{X} \geq \frac{w}{2z}$.
\end{theorem}

\begin{proof}
If there exists a birational map $\phi: \hskn \dashrightarrow \sigman$ which does not come from an isomorphism $S \ra \Sigma$ then both extremal rays of $\movb(\hskn)$ correspond to Hilbert--Chow contractions (Remark \ref{rem: unique model}). By \cite[Theorem 5.7]{bayer_macri_mmp}, $t(n-1)$ is not a square, $(n-1)X^2 - tY^2 = 1$ has no integer solutions (if $n \neq 2$) and $z \equiv \pm 1 \pmod{2(n-1)}$, $w \equiv 0 \pmod{2}$. On the other hand, assume that these conditions are satisfied. Let $a \in \left(\IZ/2t\IZ\right)^\times$, $a^2 \equiv 1 \pmod{4t}$, such that $S \cong \Sigma_a$. Since $z^2 - t(n-1)w^2 = 1$ and $w$ is even, we can consider the FM partner $\Sigma_{za}$ of $S$. As $\trans(\hskn) \cong \trans(S)$ and $\ns(\hskn) \cong \ns(S) \oplus \langle -2(n-1)\rangle$, the groups $I_a$ and $I_{za}$ can also be seen as isotropic subgroups of $A_{\trans(\hskn)} \oplus A_{\ns(\hskn)}$ (in particular, we assume $2t\eta = h \in \ns(\hskn)$). It is then immediate to check that the overlattices of $\trans(\hskn) \oplus \ns(\hskn)$ defined by these two subgroups are $H^2(\sigman_a, \IZ)$ and  $H^2(\sigman_{za}, \IZ)$ respectively. Let $\mu \in O(\ns(\hskn))$ be the isometry \eqref{eq: action}. We have that $\id \oplus \mu \in O(\trans(\hskn) \oplus \ns(\hskn))$ extends to a Hodge isometry $\psi: H^2(\sigman_a, \IZ) \ra H^2(\sigman_{za}, \IZ)$, because $\overline{\mu}(a\eta) = za\eta \in A_{\ns(\hskn)}$ (here we use the fact that $w$ is even). Notice that the discriminant group of $H^2(\hskn,\IZ)$ is generated by the class of $\frac{\delta}{2(n-1)}$ and $\overline{\mu}\left( \frac{\delta}{2(n-1)}\right) = -\frac{(n-1)wh}{2(n-1)}-\frac{z\delta}{2(n-1)} = \pm \frac{\delta}{2(n-1)}$. By \cite[Corollary 1.5.2]{nikulin}, $\psi$ extends to a Hodge isometry $H^*(\Sigma_a, \IZ) \ra H^*(\Sigma_{za},\IZ)$ between the Mukai lattices of the two K3 surfaces $\Sigma_a, \Sigma_{za}$. We conclude that $\sigman_a$ and $\sigman_{za}$ are birationally equivalent, by \cite[Corollary 9.9]{markman}.

As stated before $\Sigma_a$ and $\Sigma_{za}$ are isomorphic if and only if $a \equiv \pm za \pmod{2t}$, i.e.\ $z \equiv \pm 1 \pmod{2t}$. The isomorphism $\phi: \hskn \dashrightarrow \sigman$ is biregular if and only if the Hodge isometry $\phi^*: H^2(\sigman, \IZ) \ra H^2(\hskn, \IZ)$ maps ample classes to ample classes, i.e.\ there is only one chamber in the decomposition of $\overline{\movb(\hskn)}$. The last part of the statement follows then as in the proof of \cite[Theorem 6.4]{catt_autom_hilb}.
\end{proof}


\begin{rem} It can be readily checked that the conditions in the first part of Theorem \ref{thm: ambig} are equivalent to those of  \cite[Theorem 2.2]{mmy}. Indeed, if we write $z = 2(n-1)k + \zeta$ and $w = 2h$ for $k,h \in \IN$ and $\zeta \in \left\{ \pm 1 \right\}$, then there exist $p,q,r,s \in \IN$ such that $k = sp^2$, $k(n-1) + \zeta = rq^2$, $h = pq$, $t = rs$. In particular, $(n-1)sp^2 - rq^2 = \pm 1$, which is what is requested in \cite[Theorem 2.2]{mmy}. The FM partner $\Sigma$ of $S$ such that there exists a birational map $\hskn \dashrightarrow \sigman$ not coming from an isomorphism $S \overset{\cong}{\ra} \Sigma$ is the moduli space $M_{S}(p^2s,pqH,q^2r) \cong M_S(s,H,r)$.
\end{rem}

\section{Geometrical examples} \label{sec: constructions}

We discuss realizations of some birational automorphisms of Hilbert schemes, whose existence follows from Theorem \ref{thm: iff}. Let $S$ be a $2t$-polarized K3 surface of Picard rank one and $n \geq 2$. As already recalled, if $t=n$ then $\bir(\hskn)$ is generated by Beauville's (non-symplectic) involution \cite[\S 6]{beauville_rmks}.

If $t = 5$, the general K3 surface $S$ of degree $2t = 10$ is a transverse intersection $\textrm{Gr}(2,5) \cap \Gamma \cap Q \subset \IP^9$, where $\Gamma \cong \IP^6$ and $Q$ is a quadric. For $n=2, 3$ the Hilbert schemes $\hskn$ admit non-symplectic birational involutions, which have been explicitely described by O'Grady \cite[\S 4.3]{o'grady_epw_inv} and Debarre \cite[Example 4.12]{debarre} respectively. The next $n$ such that $\bir(\hskn) \neq \left\{ \id \right\}$ is $n=9$, where the involution is symplectic.

In the following examples we consider the case $t = 2$ for three values of $n$ where $\hskn$ is equipped with a non-symplectic birational involution. Let $\mathcal{K}_{4}$ be the moduli space of $4$-polarized K3 surfaces.


\begin{example}[$n=6$, $t = 2$]
Here the equation $P_{t(n-1)}(-1)$ is solvable with minimal solution $(3,1)$. Thus for $(S,H) \in \mathcal{K}_4$ with $\pic(S) = \IZ H$, the birational non-symplectic involution $\sigma$ which generates $\bir(S^{[6]})$ satisfies $H^2(S^{[6]},\IZ)^{\sigma^*} = \IZ\nu \cong \langle 10 \rangle$ with $\nu = 5h-3\delta$. Let $S\subset \mathbb{P}^3$ be a smooth quartic surface which does not contain any twisted cubic curves, e.g.\ $(S,H) \in \mathcal{K}_{4} \setminus D_{3,-2}$, where $H$ is the polarization given by an hyperplane section and $D_{3,-2} \subset \mathcal{K}_4$ is the Noether--Lefschetz divisor of $4$-polarized K3 surfaces $(S',H')$ for which there exists $B \in \Pic(S')$ with $(H',B)=3$, $B^2=-2$, $\langle B, H' \rangle \subset \pic(S')$ primitive. If $p_1,\ldots,p_6 \in S$ are in general linear position, there is a single rational normal curve (i.e.\ a twisted cubic) passing through them, which we denote by $C_3$. Since $C_3$ does not lie on $S$, the intersection $S\cap C_3$ is a $0$-dimensional scheme of length $12$.
The curve $C_3$ is smooth therefore the residual scheme $Z$ to $p_1+\ldots+p_6 \in S\cap C_3$ is well-defined as in \cite[\S 9]{fulton}. The birational involution of $S^{[6]}$ associates $p_1+\ldots+p_6$ to $Z$.

\end{example}

\begin{example}[$n\in \lbrace 8, 18\rbrace $, $t=2$] 
\normalfont
We have $n=tr^2$ for $r=2,3$. By Lemma \ref{lemma: t=n,4n-3} the birational non-symplectic involution $\sigma$ generating $\bir(S^{[n]})$ satisfies $H^2(S^{[n]},\IZ)^{\sigma^*} = \IZ\nu \cong \langle 2 \rangle$ with $\nu = rh-\delta$. Let $(S,H)$ be a $4$-polarized $K3$ surface with $H$ very ample and $S \hookrightarrow \IP^{n+1}$ be the composition of $\phi_{|H|}:S \ra \left| \mathcal{O}_S(H)^\vee \right| \cong \IP^3$ and $\nu_r :\IP^{3}\to \IP^{n+1}$, the Veronese map of degree $r$. The natural map $\Sym^r(H^0(S,H))\to H^0(S,rH)$ is an isomorphism, hence $\nu_r \circ \phi_{|H|} = \phi_{|rH|}$. In particular $S\subset \IP^{n+1}$ has degree $2n$ so the birational involution on $S^{[n]}$ can be described as a Beauville involution as in \cite[\S $6$]{beauville_rmks}. For the general surface, the indeterminacy locus of the automorphism is $\left\{Z\in S^{[n]}: \dim \langle Z \rangle = n-2\right\}$. If we consider $S \hookrightarrow \IP^3$ via $\phi_{|H|}$, the involution associates  $p_1+\ldots+p_n$ to its residual scheme $Z$ in the intersection between $S$ and the base locus of the linear system of degree-$r$ surfaces of $\IP^3$ through $p_1 + \ldots+p_n$.

\end{example}

For illustrative purpose, we provide the list of pairs $(n,t)$ with $2 \leq n \leq 10$, $2 \leq t \leq 10$, $t \neq n$ such that $\bir(\hskn) \neq \left\{ \id \right\}$:
\begin{equation*}
 (n,t) = (2,5), (2,10), (3,5), (3,9), (4,7), (6,2), (6,9), (6,10), (8,2), (8,4), (9,3), (9,5).
\end{equation*}
The involution which generates $\bir(\hskn)$ is symplectic for $(n,t) = (9,3), (9,5)$, non-symplectic in the other listed cases.

\appendix
\section{Pell's equations} \label{sec: prelim}
A \emph{generalized Pell's equation} is a quadratic diophantine equation $$P_r(m): X^2 - rY^2 = m$$ in the unknowns $X, Y \in \IZ$, for $r \in \IN$ and $m \in \IZ \setminus \left\{ 0\right\}$. If $m=1$, the equation is called \emph{standard}.
Two solutions $(X,Y)$, $(X',Y')$ of $P_r(m)$ are said to be equivalent if
\[ \frac{XX' - rYY'}{m} \in \IZ, \qquad \frac{XY' - X'Y}{m} \in \IZ.\]
If the equation is standard, then solutions exist and they are all equivalent.

Inside any equivalence class of solutions, the \emph{fundamental} solution $(X,Y)$ is the one with smallest non-negative $Y$, if unique. Otherwise, the smallest non-negative value of $Y$ is realized by two \emph{conjugate} solutions $(X,Y)$, $(-X,Y)$: in this case, the fundamental solution of the class will be $(X,Y)$ with $X > 0$. If $(X,Y)$ is a fundamental solution of $P_r(m)$, all other solutions $(X',Y')$ in the same equivalence class are of the form
\begin{equation} \label{eq: pell solutions}
\begin{cases}
X' = aX + rbY \\
Y' = bX + aY
\end{cases}
\end{equation}
\noindent where $(a,b)$ is a solution $P_r(1)$.

A solution $(X,Y)$ is called \emph{positive} if $X > 0$, $Y > 0$. The \emph{minimal} solution of the equation is the positive solution with smallest $X$.
We also use the expression ``minimal solution with a property $P$'' to denote the positive solution with smallest $X$ among those which satisfy the property $P$.  

Let $(z,w)$ be the minimal solution of $P_r(1)$ and $m \in \IZ \setminus \left\{ 0\right\}$. The half-open interval $\left[(\sqrt{m}, 0), (z\sqrt{m}, w\sqrt{m})\right)$ on the hyperbola $X^2-rY^2 = m$ contains exactly one solution for each equivalence class of $P_r(m)$. The solutions in this interval are all the solutions $(X,Y)$ of $P_r(m)$ such that $X > 0$ and $0 \leq \frac{Y}{X} < \frac{w}{z}$. Moreover, if $(X,Y)$ is a fundamental solution of $P_r(m)$ and $m > 0$, then:
\begin{equation} \label{eq: fundamental bounds}
0 < \left| X \right| \leq \sqrt{\frac{(z+1)m}{2}}, \qquad 0 \leq Y \leq w\sqrt{\frac{m}{2(z+1)}}.
\end{equation}

We often make use of the following lemma.

\begin{lemma} \label{lemma: congruence}
For $t \geq 1$ and $n \geq 2$ such that $t(n-1)$ is not a square, let $(z,w)$ be the minimal solution of $P_{t(n-1)}(1)$ with $z \equiv \pm 1 \pmod{n-1}$. Assume $w \equiv 0 \pmod{2}$.
\begin{enumerate}
\item[($i$)] $z \equiv 1 \pmod{2(n-1)}$ and $z \equiv 1 \pmod{2t}$ if and only if $(z,w)$ is not the minimal solution of $P_{t(n-1)}(1)$.
\item[($ii$)] $z \equiv -1 \pmod{2(n-1)}$ and $z \equiv -1 \pmod{2t}$ if and only if $P_{t(n-1)}(-1)$ is solvable.
\item[($iii$)] If $z \equiv 1 \pmod{2(n-1)}$ and $z \equiv -1 \pmod{2t}$ then the equation $(n-1)X^2 - tY^2 = -1$ has integer solutions; if $t \geq 2$ the converse also holds.
\item[($iv$)] If $z \equiv -1 \pmod{2(n-1)}$ and $z \equiv 1 \pmod{2t}$ then the equation $(n-1)X^2 - tY^2 = 1$ has integer solutions; if $n \geq 3$ the converse also holds.
\end{enumerate} 
\end{lemma}

\begin{proof}
Let $w = 2m$ for $m \in \IN$.
\begin{enumerate}
\item[($i$) -- ($ii$)] Write $z = 2(n-1)p \pm 1 = 2tq \pm 1$ for $p,q \in \IN$. Then $p((n-1)p \pm 1) = tm^2$ and $(n-1)p = tq$, hence $r \coloneqq \frac{p}{t} \in \IN$. We have $r((n-1)rt \pm 1) = m^2$ and it follows that there exist $s,u \in \IN$ such that $r = s^2$, $(n-1)tr \pm 1 = u^2$, $m = su$, hence $u^2 - t(n-1)s^2 = \pm 1$. If the sign is $+$, notice that $u < z$, therefore $(z,w)$ is not the minimal solution of $P_{t(n-1)}(1)$. Conversely, assume first that $P_{t(n-1)}(-1)$ is solvable and let $(a,b)$ be the minimal solution. Then by \cite[Lemma A.2]{debarre} the minimal solution of $P_{t(n-1)}(1)$ is $(z,w) = (2t(n-1)b^2 - 1, 2ab)$, which satisfies $z \equiv -1 \pmod{2(n-1)}$ and $z \equiv -1 \pmod{2t}$. Similarly, if the minimal solution $(u,s)$ of $P_{t(n-1)}(1)$ does not satisfy $u \equiv \pm 1 \pmod{n-1}$, then by \eqref{eq: pell solutions} we have $(z,w) = (2t(n-1)s^2+1,2us)$, hence $z \equiv 1 \pmod{2(n-1)}$ and $z \equiv 1 \pmod{2t}$.
\item[($iii$) -- ($iv$)] Write $z = 2(n-1)p \pm 1 = 2tq \mp 1$ for $p,q \in \IN$. Then $p((n-1)p \pm 1) = tm^2$ and $(n-1)p \pm 1 = tq$, hence $pq = m^2$. Since $\gcd(p,q) = 1$, there exist $s,u \in \IN$ such that $p = s^2$, $q = u^2$ and $m = su$, thus $(n-1)s^2 - tu^2 = \mp 1$. Vice versa, let $(a,b)$ be the integer solution of $(n-1)a^2 - tb^2 = \pm 1$ with smallest $a > 0$. By \cite[Lemma A.2]{debarre}, the assumption $t \geq 2$ (if the sign is $-$) or $n \geq 3$ (if the sign is $+$) implies that the minimal solution of $P_{t(n-1)}(1)$ is $(z,w) = (2(n-1)a^2 \mp 1, 2ab)$, which satisfies $z \equiv \mp 1 \pmod{2(n-1)}$ and $z \equiv \pm 1 \pmod{2t}$.\qedhere
\end{enumerate}
\end{proof}

The next two lemmas are used in Section \ref{sec: n=3}. They give bounds for the number of equivalence classes of solutions of $\pnine$ and $\ptwelve$, depending on $t \in \IN$.

\begin{lemma} \label{lemma: pnine}
If $t \equiv 1 \pmod{3}$ or $t \equiv 3,6 \pmod{9}$, then all solutions of $\pnine$ are equivalent to $(X,Y) = (3,0)$. If $t \equiv 0 \pmod{9}$, then $\pnine$ has either one, two or three classes of solutions. If $t \equiv 2 \pmod{3}$, then $\pnine$ has either one or three classes of solutions.
\end{lemma}

\begin{proof}
If $t = 3q+1$, for $q \in \IN_0$, then $X^2 \equiv 2Y^2 \pmod{3}$, hence all solutions $(X,Y)$ are of the form $(3X',3Y')$, with $(X')^2 - 8t(Y')^2 = 1$. This is now a standard Pell's equation, which has only one class of solutions, thus the same holds for $\pnine$.

If $t = 3q$ for $q \in \IN$ and $(q,3) = 1$, then $X = 3X'$ for some $X' \in \IZ$ such that $3(X')^2 - 8qY^2 = 3$. Since $3 \nmid q$ we need $(X,Y) = (3X', 3Y')$ with $(X')^2 - 24q(Y')^2 = 1$. There exists only one class of solutions $(X',Y')$ for this standard Pell's equation, thus also the solutions of $\pnine$ form a single class.

Assume now that $t = 9q$ for $q \in \IN$. Then a solution $(X,Y)$ of $\pnine$ is of the form $(X,Y) = (3X',Y)$ with $(X')^2 - 8qY^2 = 1$. Let $(z,w)$ be the minimal solution of $\pone : z^2 - 8tw^2 = 1$. Notice that the solutions of $\pone$ are the pairs $(X', \frac{Y}{3})$ for all solutions $(X',Y)$ of  $(X')^2 - 8qY^2 = 1$ such that $Y \equiv 0 \pmod{3}$. Let $(a,b)$ be the minimal solution of $(X')^2 - 8qY^2 = 1$. Since this is a standard Pell's equation, by \eqref{eq: pell solutions} its next two solutions (for increasing values of $X'$) are $(a^2 + 8qb^2,2ab)$ and $(a^3+24qab^2, 8qb^3+3a^2b)$. We observe that one among these first three positive solutions $(X',Y)$ has $Y \equiv 0 \pmod{3}$. The positive solution $(X',Y)$ with this property and smallest $X'$ is therefore equal to $(z, 3w)$, thus the corresponding solution $(X,Y) = (3X',Y)$ of $\pnine$ satisfies $\frac{Y}{X} = \frac{w}{z}$, i.e.\ it is the first positive solution in the same equivalence class of $(3,0)$. We conclude that $\pnine$ has either one, two or three classes of solutions.

Finally, assume that $t \equiv 2 \pmod{3}$. Let $(z,w)$ be the minimal solution of $\pone$ and let $(u_1,v_1)$, $(u_2,v_2)$ be two positive solutions of $\pnine$ such that, for $i=1,2$: 
\begin{equation}\label{eq: fundamental}
0 < u_i \leq 3\sqrt{\frac{z+1}{2}}, \quad 0 < v_i \leq \frac{3w}{\sqrt{2(z+1)}}.
\end{equation}
By \eqref{eq: fundamental bounds}, this is equivalent to asking that either $(u_i, v_i)$ or $(-u_i, v_i)$ is a fundamental solution of $\pnine$, different from $(3,0)$. Thus $u_1,v_1,u_2,v_2$ are not divisible by three. From $u_1^2 - 8tv_1^2 = 9$ and $u_2^2 - 8tv_2^2 = 9$ we get 
\begin{equation}\label{eq: pm mod9}
u_1v_2 \equiv \pm u_2v_1 \pmod{9}, \quad u_1u_2 \equiv \pm 8t v_1v_2 \pmod{9}
\end{equation}
where the signs in the two congruences coincide. If we now multiply $u_1^2 - 8tv_1^2 = 9$ and $u_2^2 - 8tv_2^2 = 9$ member by member, we obtain
\begin{equation}\label{eq: reduce to P1}
\left( \frac{u_1u_2 \mp 8tv_1v_2}{9} \right)^2 - 8t\left( \frac{u_1v_2 \mp u_2v_1}{9} \right)^2 = 1
\end{equation}
\noindent where by \eqref{eq: pm mod9} the two squares in the LHS term are integers. If we assume that the pairs $(u_1,v_1)$ and $(u_2,v_2)$ are distinct, then $u_1v_2 \mp u_2v_1 \neq 0$. Since $(z,w)$ is the minimal solution of $\pone$, from \eqref{eq: reduce to P1} we have $\left| u_1v_2 \mp u_2v_1 \right| \geq 9w$. However, from \eqref{eq: fundamental} we compute $\left| u_1v_2 \mp u_2v_1 \right| < 9w$, which is a contradiction. We conclude $u_1 = u_2$, $v_1 = v_2$. Thus, there are at most three classes of solutions for $\pnine$: the class of $(3,0)$ and possibly the classes of $(u_1,v_1)$ and $(-u_1,v_1)$. Notice that the latter two classes are always distinct: in order for them to coincide we would need 
\[ \frac{u_1^2 + 8tv_1^2}{9} \in \IZ, \qquad \frac{2u_1v_1}{9} \in \IZ\]
\noindent which does not happen, since $\gcd(u_1,3) = \gcd(v_1,3) = 1$. Hence, the equation has either one or three classes of solutions.
\end{proof}

In an entirely similar way one can prove the following.

\begin{lemma} \label{lemma: ptwelve}
If $t \equiv 3 \pmod{18}$, then $\ptwelve$ is either not solvable or it has one class of solutions. If $t \equiv 5,11,17 \pmod{18}$, then $\ptwelve$ is either not solvable or it has two classes of solutions. In all other cases, $\ptwelve$ is not solvable.
\end{lemma}

%
%
%

All cases in Lemma \ref{lemma: pnine} and Lemma \ref{lemma: ptwelve} occur, for suitable values of $t$.

\providecommand{\bysame}{\leavevmode\hbox to3em{\hrulefill}\thinspace}
\providecommand{\MR}{\relax\ifhmode\unskip\space\fi MR }
\providecommand{\MRhref}[2]{%
  \href{http://www.ams.org/mathscinet-getitem?mr=#1}{#2}
}
\providecommand{\href}[2]{#2}

\end{document}